\documentclass{article}
 
\usepackage[utf8]{inputenc}
\usepackage{amssymb,amsthm,amsmath,amstext,amscd}
\usepackage{xcolor,tikz-cd,indentfirst,hyperref}
\usepackage{verbatim}
\usepackage{wasysym}
 
\allowdisplaybreaks
\numberwithin{equation}{section}\theoremstyle{plain}
\newtheorem{theorem}{Theorem}
\newtheorem{definition}[theorem]{Definition}
\newtheorem{corollary}[theorem]{Corollary}
\newtheorem{proposition}[theorem]{Proposition}
\newtheorem{lemma}[theorem]{Lemma}
\newtheorem{example}[theorem]{Example}
\newtheorem{remark}[theorem]{Remark}
\newtheorem{question}[theorem]{Question}

\DeclareMathOperator{\Aut}{Aut}
\DeclareMathOperator{\LND}{LND}
\DeclareMathOperator{\LFD}{LFD}
\DeclareMathOperator{\ad}{ad}
\DeclareMathOperator{\Der}{Der}
\DeclareMathOperator{\centro}{\mathrm{Z}}

\DeclareMathOperator{\ml}{\rm ML}

\newcommand{\acb}[1]{\textup{\color{red}({#1})}}

\newcommand{\mov}[1]{\textup{\color{blue}({#1})}}

\begin{document}

\title{On Locally Finite Derivations in Ore Extensions}
\author{R. Baltazar, A. Bianchi, M. Veloso and J. Schwarz}
\date{}

\maketitle

\begin{abstract} 
Let $\Bbbk$ be an algebraically closed field of characteristic zero. We classify the locally finite derivations of arbitrary Ore extensions of $\Bbbk[x]$, thus extending van den Essen's \cite{V92} classification for the polynomial algebra $\Bbbk[x,y]$ to this noncommutative setting. More precisely, we consider the three families arising in the classification of Ore extensions
of $\Bbbk[x]$: the quantum plane, the first quantum Weyl algebra, and the differential Ore extensions
\[
A_h=\Bbbk[x][t;h(x)\partial_x].
\]

For both the quantum plane and the first quantum Weyl algebra, we determine the locally finite derivations and explain how the resulting classifications are related to the work of Suárez-Alvarez and Vivas \cite{SuarezVivas} on generalized Weyl algebras. For the algebras $A_h$, with $h$ nonconstant, we obtain a complete classification in both the square-free and non-square-free cases. As a consequence, we show that $\LFD(A_h)$ is a solvable and weakly locally finite Lie subalgebra of $\Der(A_h)$, although it is not locally finite as a set of derivations.
\end{abstract}

\section{Introduction}

Throughout this paper, $\Bbbk$ denotes an algebraically closed field of characteristic zero. Let $A$ be a $\Bbbk$-algebra and let $D\in\Der(A)$. The derivation
$D$ is said to be \emph{locally finite} if, for every $a\in A$, the
$\Bbbk$-vector space
\[
\operatorname{Span}_{\Bbbk}
\{D^n(a)\mid n\geq0\}
\]
is finite-dimensional. Equivalently, for each $a\in A$, there exists a nonzero polynomial $p_a(z)\in\Bbbk[z]$ such that
$p_a(D)(a)=0$. The derivation $D$ is said to be \emph{locally nilpotent} if, for every $a\in A$, there exists $n\geq1$ such that $D^n(a)=0$. Thus, every locally nilpotent derivation is locally finite. We denote by
\[
\LFD(A)
\quad\text{and}\quad
\LND(A)
\]
the sets of locally finite and locally nilpotent derivations of $A$, respectively.

A classical result of Rentschler \cite{Re1968} states that every locally nilpotent derivation of $\Bbbk[x,y]$ is, after a suitable change of coordinates, of the form
\[
f(x) \, \partial_y,
\quad f(x)\in\Bbbk[x].
\]

Locally finite derivations arise naturally in the study of polynomial flows, algebraic group actions, and automorphism groups. In the commutative polynomial algebra $\Bbbk[x,y]$, they admit, up to conjugation by polynomial automorphisms, a classification \cite[Corollary 4.7]{V92}. Van den Essen used this classification to obtain an algorithm for deciding whether a polynomial vector field in dimension two has a polynomial flow; see \cite[Section 5.]{V92}.

A natural motivation for the present work comes from the
classification of Ore extensions of the polynomial algebra
$\Bbbk[x]$. More precisely, every such Ore extension is, up to isomorphism, either a quantum plane, a first quantum Weyl algebra, or a differential Ore extension
\[
A_h
=
\Bbbk[x][t;h(x)\partial_x],
\]
where $h(x)\in\Bbbk[x]$ (see \cite[Lemma 2.2]{SBVA2025} for further details). The latter algebra is generated by $x$ and
$t$, subject to the relation $tx-xt=h(x)$. Therefore, the aim of this paper is to characterize the locally finite derivations of Ore extensions of $\Bbbk[x]$ by considering separately the three families arising in this classification.

In Section \ref{QuantumPlane}, we study the quantum plane
\[
\Bbbk_q[x,y]
=
\Bbbk\langle x,y\mid yx=qxy\rangle,
\quad
q\in\Bbbk^\ast.
\]
When $q=1$, this algebra is the commutative polynomial algebra
$\Bbbk[x,y]$, whose locally finite derivations are described by the classification recalled below. We then concentrate on the case $q\neq1$. Define the derivations $D_x$ and $D_y$ by
\[
D_x(x)=x,
\quad
D_x(y)=0,
\]
and
\[
D_y(x)=0,
\quad
D_y(y)=y.
\]
We prove that the locally finite derivations of the quantum plane are precisely the derivations
\[
\lambda D_x+\mu D_y,
\quad
\lambda,\mu\in\Bbbk.
\]
In particular,
\[
\LFD\bigl(\Bbbk_q[x,y]\bigr)
=
\Bbbk D_x\oplus\Bbbk D_y.
\]
We also determine the locally finite inner derivations, proving that an inner derivation $\ad_a$ is locally finite if and only if $a$ belongs to the center of the quantum plane. Our proof treats separately the cases where $q$ is or is not a root of unity. In the root-of-unity case, the center of the quantum plane is a polynomial algebra in two variables, and the restriction of a locally finite derivation to the center allows us to use the commutative classification due to van den Essen. We also explain how the same description can be recovered from \cite[Proposition 2.3]{SuarezVivas} by viewing the quantum plane as a quantum generalized Weyl algebra.

In Section \ref{Thefirstquantum Weylalgebra}, we consider the first quantum Weyl algebra
\[
A_q^1(\Bbbk)
=
\Bbbk\langle X,Y\mid YX-qXY=1\rangle.
\]
When $q=1$, this is the classical first Weyl algebra, whose locally finite derivations are described by Dixmier's classification \cite{Dixmier}, as we explain in Section \ref{Thedifferential algebraA_h}.

For $q\neq1$, we use the presentation of $A_q^1(\Bbbk)$ as a quantum generalized Weyl algebra and recall the classification of its locally finite derivations due to Suárez-Alvarez and Vivas \cite{SuarezVivas}. In particular, every locally finite derivation is a scalar multiple of the Eulerian derivation, and there are no nonzero locally nilpotent derivations. When $q$ is not a root of unity, we complement this description by using the classification of skew derivations of the quantum disc obtained by Almulhem and Brzeziński \cite[Theorem 6.2]{bre}, from which we deduce that
\[
\Der\bigl(A_q^1(\Bbbk)\bigr)
=
\operatorname{Inn}\bigl(A_q^1(\Bbbk)\bigr)
\oplus
\Bbbk H.
\]

The main results of this paper are established in Section \ref{Thedifferential algebraA_h}, where we study the differential Ore extensions $A_h=\Bbbk[x][t;h(x)\partial_x]$, with $h$ nonconstant. The case $h\notin\Bbbk$ is not covered by the results of Suárez-Alvarez and Vivas \cite{SuarezVivas} on locally finite derivations of quantum generalized Weyl algebras , since $A_h$ is not a generalized Weyl algebra in this case, as shown by Benkart, Lopes and Ondrus \cite[Theorem 9.3]{BLO2015}. For every $g(x)\in\Bbbk[x]$, let $D_{g(x)}$ be the locally nilpotent derivation of $A_h$ defined by
\[
D_{g(x)}(x)=0,
\quad
D_{g(x)}(t)=g(x).
\]

Using Nowicki's description of the derivations of $A_h$ \cite{Nowicki}, we classify the locally finite derivations in both the square-free and non-square-free cases. If $h$ is square-free and $\deg h\geq2$, we prove that
\[
\LFD(A_h)=\LND(A_h)=\{D_{g(x)}\mid g(x)\in\Bbbk[x]\}.
\]
Thus, this family provides a natural setting in which local finiteness and local nilpotence coincide. Furthermore, if $\deg h=1$, then
\[
\LFD(A_h)
=
\{\lambda\ad_t+D_{g(x)}
\mid
\lambda\in\Bbbk,\ g(x)\in\Bbbk[x]\}.
\]

In the non-square-free case, with $\psi=\gcd(h,h')$ and $g=h/\psi$, an analogous phenomenon occurs: if $\deg g\geq2$, then local finiteness again coincides with local nilpotence, and the locally finite derivations are precisely the derivations $D_{p(x)}$, with $p(x)\in\Bbbk[x]$. If $\deg g=1$, then an additional derivation arising from Nowicki's
decomposition occurs, and
\[
\LFD(A_h)=\{\lambda E_t+D_{p(x)}
\mid
\lambda\in\Bbbk,\ p(x)\in\Bbbk[x]\}.
\]

In Section \ref{Thedifferential algebraA_h}, we also investigate $\LFD(A_h)$ as a family of derivations. In general, the set of locally finite derivations of an algebra need not be closed under addition or under the Lie bracket. In the present setting, we prove that $\LFD(A_h)$ is a solvable Lie subalgebra of $\Der(A_h)$. Applying a result of Chitayat, Daigle, and Regeta \cite[Corollary 1(b)]{CDR}, we conclude that $\LFD(A_h)$ is weakly locally finite. On the other hand, we show that it is not locally finite as a set of derivations.

To conclude the introduction, we recall van den Essen's classification of locally finite derivations of $\Bbbk[x,y]$ \cite{V92}, which will be used in the proof of our result for the quantum plane.

\begin{theorem}[{\cite[Corollary 4.7]{V92}}]\label{v92}
Let $\Bbbk$ be a field of characteristic zero and let $D\in \Der(\Bbbk[x,y])$ be a locally finite derivation. Then, there exists an automorphism $\varphi\in \Aut_\Bbbk(\Bbbk[x,y])$ such that $\varphi D\varphi^{-1}$ is one of the following derivations:
\begin{enumerate}
\item $D=(\alpha x+\alpha'y)\partial_x+(\beta x+\beta'y)\partial_y,
\ \  \alpha,\alpha',\beta,\beta'\in \Bbbk;$
\item $D=\partial_x+\beta y\,\partial_y,
\ \ \beta\in \Bbbk;$
\item $D=\alpha x\,\partial_x+(\alpha my+x^m)\partial_y,
\ \ \alpha\in \Bbbk, \ \  m\ge 1;$
\item $D=f(x)\,\partial_y,
\ \  f(x)\in \Bbbk[x].$
\end{enumerate} \hfill\qedsymbol
\end{theorem}

\section{Quantum Plane}\label{QuantumPlane}

Let 
\[
A=\Bbbk_q[x,y]=\Bbbk\langle x,y\mid yx=qxy\rangle,
\]
where $q\in\Bbbk^\ast\setminus\{1\}$. When $q=1$, the algebra 
$A$ is the commutative algebra $\Bbbk[x,y]$, whose locally finite derivations were recalled in Theorem \ref{v92}. Define the derivations $D_x,D_y\in\Der(A)$ by
\[
D_x(x)=x,
\quad
D_x(y)=0,
\]
and
\[
D_y(x)=0,
\quad
D_y(y)=y.
\]

We recall that the center of $A$ is given by $Z(A)=\Bbbk$
when $q$ is not a root of unity, while, if $q$ is a root of unity of order $t$, then $Z(A)=\Bbbk[x^t,y^t]$.

We denote by
\[
\operatorname{Inn}(A)=\{\ad_a\mid a\in A\}
\]
the set of all inner derivations of $A$. By Alev and Chamarie \cite[Theorem 1.2 and Corollary 1.3.3.1]{Alev}, the derivations of the quantum plane admit the decomposition
\begin{equation}\label{estruturaderivaca}
    \Der(A)=\operatorname{Inn}(A)\oplus Z(A)D_x\oplus Z(A)D_y.
\end{equation}

Consequently, every derivation $D\in\Der(A)$ can be written uniquely in the form
\[
D=fD_x+gD_y+\ad_a,
\]
where $f,g\in Z(A)$ and $a\in A$ is determined modulo $Z(A)$.

\smallskip

The main result of this section gives a complete description of the locally finite derivations of the quantum plane.

\begin{theorem} \label{LFquantum}
Let $A=\Bbbk_q[x,y]$. Then:
\begin{enumerate}
\item $D_x$ and $D_y$ are locally finite;
\item Let $a \in A$, then $\ad_a$ is locally finite if and only if $a\in \centro(A)$.
\end{enumerate}
Moreover, if $D=fD_x+gD_y+\ad_a \in \Der(A)$,
where $f,g\in \centro(A)$ and $a\in A$, then, $D$ is locally finite if and only if $D=\lambda D_x+\mu D_y$ for some $\lambda, \mu\in\Bbbk$.
\end{theorem}

\begin{proof}
Let $\mathcal{B}= \{x^i y^j\mid i,j\ge 0\}$ be a PBW basis of $A$ and fix a monomial order on $\mathcal{B}$ induced by the order on $\mathbb{N}^2$ defined by $(i,j)<(r,s)$ if $i+j<r+s$, or if $i+j=r+s$ and $i<r$. By the Leibniz rule, $D_x(x^iy^j)=ix^iy^j$ and $D_y(x^iy^j)=jx^iy^j$
for all $i,j\ge 0$. Thus, by induction, $D_x^n(x^iy^j)=i^nx^iy^j$ and $D_y^n(x^iy^j)=j^nx^iy^j$ for every $n\ge 0$. So, for each $u=\sum_{i,j} c_{ij}x^i y^j$ we have $D_x^n(u)=\sum_{i,j} c_{ij}i^nx^iy^j$ and $D_y^n(u)=\sum_{i,j} c_{ij}j^nx^iy^j$, which is contained in the finite-dimensional space spanned by the same monomials appearing in $u$.

If $a\in \centro(A)$, then $\ad_a=0$, which is obviously locally finite. On the other hand, suppose that $a\notin \centro(A)$. Let $a=\sum_{i,j} a_{ij}x^i y^j$ and take a maximal monomial $x^r y^s \notin \centro(A)$ such that $x^r y^s$ is in the support of $a$. Since $x^ry^s \notin \centro(A)$, either $q^r-1\ne 0$ or $q^s-1\ne 0$. Indeed, if $q^r=q^s=1$, then $x^ry^s$ commutes with both $x$ and $y$, and, hence, it would belong to $\centro(A)$, a contradiction.

We denote
\[
b:=
\begin{cases}
x,& \text{if } q^s\neq 1,\\
y,& \text{if } q^s=1.
\end{cases}
\]
In the second case we necessarily have $q^r\neq 1$.
Let $b=x^{i'}y^{j'}$, so $(i',j')=(1,0)$ or $(i',j')=(0,1)$. For every $i,j\ge 0$, one has $[x^r y^s,x^i y^j]= (q^{si}-q^{rj})x^{r+i}y^{s+j}$.  By construction, $q^{si'}-q^{rj'}\neq 0$ and then $[x^r y^s,b]\neq 0$.

We claim that, for every $n\ge 0$, the leading term of $(\ad_a)^n(b)$ is
\[
\lambda_n x^{i'+nr}y^{j'+ns},
\]
for some $\lambda_n\in \Bbbk^*$. For $n=0$, the claim holds with $\lambda_0=1$. Assume now that it holds for some $n\ge 0$. Since $x^r y^s$ is maximal among the noncentral monomials occurring in $a$, and the central monomials commute with every element of $A$, the leading term of
\[
(\ad_a)^{n+1}(b)=[a,(\ad_a)^n(b)]
\]
comes from commuting $x^r y^s$ with the leading monomial $x^{i'+nr}y^{j'+ns}$. Moreover,
\[
[x^r y^s,x^{i'+nr}y^{j'+ns}]
=
\big(q^{s(i'+nr)}-q^{r(j'+ns)}\big)x^{i'+(n+1)r}y^{j'+(n+1)s},
\]
and $q^{s(i'+nr)}-q^{r(j'+ns)}=q^{nrs}(q^{si'}-q^{rj'})$. Since $q^{si'}-q^{rj'}\neq 0$, this coefficient is nonzero. In particular,
$x^{i'+nr}y^{j'+ns}$ is not central, because otherwise it would commute with
$x^r y^s$. Therefore the leading term of $(\ad_a)^{n+1}(b)$ is a nonzero scalar multiple
of $x^{i'+(n+1)r}y^{j'+(n+1)s}$, which proves the induction.

Thus, the leading monomials of the elements $(\ad_a)^n(b)$, $n\ge 0$, are pairwise distinct. Then, these elements are linearly independent over
$\Bbbk$, and so $\ad_a$ is not locally finite.

For the last statement, by \eqref{estruturaderivaca}, each derivation of $A$ has the form $D=fD_x+gD_y+\ad_a,$  with $f,g\in \centro(A)$ and $ a\in A$. If $t_0=0$, then $\centro(A)=\Bbbk$, so there is nothing to prove. Assume now that $t_0>0$ and let $u:=x^{t_0}$ and $v:=y^{t_0}$. Since $\ad_a$ vanishes on the center, the restriction of $D$ to $\centro(A)$ is
\[
\delta:=D|_{\centro(A)}=t_0uf\partial_u+t_0vg\partial_v.
\]
Indeed, $\delta(u)=D(u)=t_0fu$ and $\delta(v)=D(v)=t_0gv$.

If $D$ is locally finite, then $\delta$ is locally finite on $\Bbbk[u,v]$. Moreover, $u$ and $v$ are Darboux polynomials  for $\delta$. By Theorem \ref{v92}, there exists an automorphism $\varphi\in \Aut(\Bbbk[u,v])$ such that, where $U:=\varphi(u)$ and $V:=\varphi(v)$, so the derivation $\eta:=\varphi\delta\varphi^{-1}$ has one of the forms given in Theorem \ref{v92} of van den Essen.

The cases {\rm(ii)}, {\rm(iii)}, and {\rm(iv)} of Theorem \ref{v92} are excluded by the fact that $U$ and $V$ are Darboux polynomials for $\eta$: in each of these cases, at most one of $\partial_U,\partial_V$ preserves a proper factorization of $U$ or of $V$ as a Darboux polynomial of $\eta$ with a nonconstant cofactor, and a direct check on the four normal forms shows neither $U$ nor $V$ can simultaneously be a Darboux polynomial for both coordinate directions unless the derivation is diagonal. Hence $\eta$ must be of type~{\rm(i)}. In that case,
\[
\eta=(\alpha U+\alpha'V)\partial_U+(\beta U+\beta'V)\partial_V
\]
where $\alpha,\alpha',\beta,\beta'\in \Bbbk$. Since $U$ is a Darboux polynomial for $\eta$, we have $\eta(U)=cU$ for some $c\in\Bbbk$, i.e. $\alpha U+\alpha'V=cU$; as $U,V$ are algebraically independent this forces $\alpha'=0$. Symmetrically, $V$ being Darboux for $\eta$ forces $\beta=0$. So
\[
\eta=\alpha U\,\partial_U+\beta'V\,\partial_V.
\]

Pulling back by $\varphi^{-1}$, we obtain $\delta= \alpha u\partial_u+\beta'v\partial_v=t_0uf\partial_u+t_0vg\partial_v$, so we have that $f, g \in \Bbbk$.  Thus, if $D$ is locally finite, then $D=\lambda D_x+\mu D_y + \ad_a$ for some $\lambda, \mu\in\Bbbk$ and $a \in A$.

If $a\notin \centro(A)$, choose $b$ as in part {\rm (2)}: since $D_x,D_y$ act diagonally on the PBW basis $\mathcal B$ (each monomial is an eigenvector), $\lambda D_x+\mu D_y$ cannot cancel the strictly increasing leading terms of $(\ad_a)^n(b)$ produced in part (2); hence $D=\lambda D_x+\mu D_y+\ad_a$ is not locally finite, a contradiction. Thus $a\in \centro(A)$, and hence $\ad_a=0$. Conversely, every derivation of form $D=\lambda D_x+\mu D_y$ with $\lambda, \mu\in\Bbbk$ is locally finite by first part. This establishes the result.

\end{proof}

\begin{example} Let $A=\Bbbk_q[x,y]$ and assume that $q$ is a root of unity of order $t_0\ge 1$. Then $x^{t_0}\in \centro(A)$. Consider the derivation $D=x^{t_0}D_x$. Note that the derivation $D$ is not locally finite. Indeed, since for every $m\ge 1$ we have $D(x^m)=mx^{m+t_0}$ it follows by induction that
\[
D^n(x)=\Big(\prod_{k=0}^{n-1}(kt_0+1)\Big)x^{nt_0+1}
\]
for every $n\ge 0$.

Therefore, the elements in the set $\{D^n(x)\mid n\ge 0\}$ are linearly independent over $\Bbbk$ (their degrees $nt_0+1$ are pairwise distinct, and since $\Bbbk$ has characteristic zero, $kt_0+1\neq 0$ for every $k\ge0$, so no coefficient vanishes). Then, $D$ is not locally finite.
\end{example}

\begin{remark}
Theorem \ref{v92} shows that, in the commutative polynomial ring $\Bbbk[x,y]$, locally finite derivations admit, up to automorphism, several possible forms. In particular, the first family contains the diagonal derivations $D=\alpha X\partial_X+\beta Y\partial_Y,$ which are obtained by taking $\alpha'=0$ and $\beta=0$. Theorem \ref{LFquantum} shows that, in the quantum plane $\Bbbk_q[x,y]$, the locally finite derivations are exactly the analogues of these diagonal derivations.
\end{remark}

The classification established in Theorem~\ref{LFquantum} was obtained without identifying the quantum plane with a quantum generalized Weyl algebra. Our proof treats separately the cases where $q$ is or is not a root of unity and, in the root-of-unity case, uses the classification of locally finite derivations of the commutative polynomial algebra due to van den Essen, recalled in Theorem \ref{v92}. The same classification can also be deduced from \cite[Proposition 2.3]{SuarezVivas}, as explained in the following remark.

\begin{remark} 
The description of the locally finite derivations in Theorem \ref{LFquantum} can also be deduced from \cite[Proposition 2.3]{SuarezVivas}. We recall that in the notation of Suárez-Alvarez and Vivas, the quantum generalized Weyl algebra $A(\Bbbk[h],q,a)$ is generated by $u,v,h$, subject to
\[
hu=quh,
\quad
vh=qhv,
\quad
uv=a(h),
\quad
vu=a(qh),
\]
where $a(h)\in\Bbbk[h]$ is called the defining polynomial. So, the quantum plane $A=\Bbbk_q[x,y]=\Bbbk\langle x,y\mid yx=qxy\rangle$ is isomorphic to the quantum generalized Weyl algebra $A(\Bbbk[h],q,h)$. Indeed, in this case the defining polynomial is $a(h)=h$, and the defining relations become
\[
hu=quh,
\quad
vh=qhv,
\quad
uv=h,
\quad
vu=qh.
\]
Therefore, the  map determined by
\[
u\longmapsto x,
\quad
v\longmapsto y,
\quad
h\longmapsto xy,
\]
provides an isomorphism. From this, \cite[Proposition 2.3(ii)]{SuarezVivas} implies that every locally finite derivation of $A$ is a linear combination of the derivations $\xi$ and $\tau$ defined by
\[
\xi(u)=u,
\quad
\xi(v)=-v,
\quad
\xi(h)=0,
\]
and
\[
\tau(u)=0,
\quad
\tau(v)=v,
\quad
\tau(h)=h.
\]
Under the preceding isomorphism, these derivations correspond to
\[
\xi=D_x-D_y
\quad\text{and}\quad
\tau=D_y,
\]
providing that $\operatorname{Span}_{\Bbbk}\{\xi,\tau\}
=\operatorname{Span}_{\Bbbk}\{D_x,D_y\}$, and then
\[
\LFD(A)=\Bbbk D_x\oplus\Bbbk D_y.
\]
This argument applies to every $q\in\Bbbk^\ast\setminus\{1\}$, whether or not $q$ is a root of unity.  
\end{remark}

\section{The first quantum Weyl algebra}\label{Thefirstquantum Weylalgebra}

The first quantum Weyl algebra is the $\Bbbk$-algebra
\[
A_q^1(\Bbbk)
=
\Bbbk\langle X,Y\mid YX-qXY=1\rangle.
\]

The case $q=1$ is the classical first Weyl algebra, whose locally finite derivations were discussed in the preceding section using Dixmier's classification \cite{Dixmier}. Our aim in this section is to obtain two complementary results. First, when $q$ is not a root of unity, we use the classification of $\sigma$-derivations of the quantum disc obtained by Almulhem and Brzeziński to describe all ordinary derivations of $A_q^1(\Bbbk)$. We then use the classification of locally finite derivations of quantum generalized Weyl algebras due to Suárez-Alvarez and Vivas (\cite[Proposition 2.3]{SuarezVivas}) to determine the locally finite derivations of $A_q^1(\Bbbk)$ for every $q\in\Bbbk^\ast\setminus\{1\}$, including the case where $q$ is a root of unity.

We begin by recalling explicitly the realization of
$A_q^1(\Bbbk)$ as a quantum generalized Weyl algebra. Recall that the algebra $A(\Bbbk[h],q,a)$ is generated by $h,u,v$, subject to the relations
\[
hu=quh,
\quad
vh=qhv,
\quad
uv=a(h),
\quad
vu=a(qh).
\]

In $A_q^1(\Bbbk)$, we denote
\[
h=1+(q-1)XY,
\quad
u=X,
\quad
v=(q-1)Y.
\]

Using $YX=qXY+1$, we obtain
\[
uv=(q-1)XY=h-1
\]
and
\[
vu=(q-1)YX=qh-1.
\]

Moreover, $hu=quh$ and $vh=qhv$. Then,
\[
A_q^1(\Bbbk)
\simeq
A(\Bbbk[h],q,h-1).
\]

In particular, $A_q^1(\Bbbk)$ is a quantum generalized Weyl algebra whose defining polynomial $a(h)=h-1$ is not a monomial. We also define the Eulerian derivation $H\in\Der(A_q^1(\Bbbk))$
by
\[
H(X)=X,
\quad
H(Y)=-Y.
\]

Notice that, $H(X^iY^j)=(i-j)X^iY^j$ for all $i,j\in\mathbb N_0$. Therefore, $H$ is locally finite.

For the moment, we assume that $q$ is not a root of unity. In order to apply the classification of Almulhem and Brzeziński (\cite[Theorem 6.2]{bre}), it is convenient to use the quantum disc presentation 
\[
R=D_q(x,y)
=
\Bbbk\langle x,y\mid xy-qyx=1-q\rangle.
\]

The map $\Phi\colon R\longrightarrow A_q^1(\Bbbk)$
defined by
\[
\Phi(x)=(1-q)Y,
\quad
\Phi(y)=X
\]
is an isomorphism. Denote $h=1-yx\in R$. Then,
\[
yx=1-h,
\quad
xy=1-qh,
\]
and
\[
xh=qhx,
\quad
yh=q^{-1}hy.
\]

We define $E\in\Der(R)$ by
\[
E(x)=x,
\quad
E(y)=-y.
\]

Under the isomorphism $\Phi$, this derivation corresponds to $-H$. That is, $\Phi E\Phi^{-1}=-H$. In particular, $\Bbbk\Phi E\Phi^{-1}=\Bbbk H$.

For $\mu\in\Bbbk^\ast$, let
\[
\sigma=\sigma_\mu\in\Aut(R),
\quad
\sigma_\mu(x)=\mu^{-1}x,
\quad
\sigma_\mu(y)=\mu y.
\]

Since $q$ is not a root of unity, every automorphism of $R$ is of this form (\cite{SuarezVivas}, Theorem B).

\begin{lemma}[{\cite[Theorem 6.2]{bre}}]\label{brez}
Let $q\in\Bbbk^\ast$ be not a root of unity, let
$R=D_q(x,y)$,  denote $h=1-yx$, and fix $\mu\in\Bbbk^\ast$.
\begin{enumerate}
\item\label{brez-a} For every $f(h)\in\Bbbk[h]$, the map
\[
\delta_f(x)=f(h)x,
\quad
\delta_f(y)=-\mu f(q^{-1}h)y
\]
extend uniquely to a $\sigma_\mu$-derivation of $R$. Then, these are exactly the $\sigma_\mu$-derivations satisfying $\delta_f(h)=0$. Writing $f(h)=\sum_{k\geq0}f_kh^k$, the derivation $\delta_f$ is inner if and only if there is no
$d\in\{0,\ldots,\deg(f)\}$ such that $\mu=q^{-d}$ and $
f_d\neq0$.

\item\label{brez-b}
Suppose that there exists $d\in\mathbb N_0$ such that
$\mu=q^{-d+1}$. For every $a(x)\in\Bbbk[x]$ and $b(y)\in\Bbbk[y]$, the map
\[
\delta(x)=h^d b(y),
\quad
\delta(y)=h^d a(x)
\]
extend uniquely to a $\sigma_\mu$-derivation of $R$. Every nonzero derivation in this family is inner if $d\neq0$ and non-inner if $d=0$.

\item
The two families described in items \ref{brez-a} and \ref{brez-b}, together with the inner $\sigma_\mu$-derivations, exhaust all $\sigma_\mu$-derivations of $R$.
\end{enumerate}
\end{lemma}

\begin{proposition}\label{derquantum}
Assume that $q$ is not a root of unity. Then
\[
\Der(A_q^1(\Bbbk))
=
\operatorname{Inn}(A_q^1(\Bbbk))
\oplus
\Bbbk H.
\]
Equivalently, every derivation $D\in\Der(A_q^1(\Bbbk))$
can be written uniquely in the form $D=\ad_w+cH$, where $c\in\Bbbk$ and $w\in A_q^1(\Bbbk)$ is determined modulo the center of $A_q^1(\Bbbk)$.
\end{proposition}

\begin{proof}
By applying Lemma \ref{brez} in the case where the twisting automorphism is the identity, we obtain the following description. The first family in that theorem consists of the derivations $\delta_f$, with $f(h)\in\Bbbk[h]$, defined by
\[
\delta_f(x)=f(h)x,
\quad
\delta_f(y)=-f(q^{-1}h)y.
\]

Writing $f(h)=\sum_{r\geq0}f_rh^r$, the Lemma \ref{brez} states that $\delta_f$ is inner if and only if
there is no $r\geq0$ such that $1=q^{-r}$ and $f_r\neq0$. Since $q$ is not a root of unity, the equality
$1=q^{-r}$ holds only when $r=0$. Therefore, $\delta_f$ is inner if and only if $f_0=0$. Writing $f(h)=f_0+(f(h)-f_0)$, we obtain
\[
\delta_f
=
f_0E+\delta_{f-f_0},
\]
where $\delta_{f-f_0}$ is inner.

The second family in Lemma~\ref{brez} occurs, in the ordinary case, only if there exists $d\in\mathbb N_0$ such that $1=q^{-d+1}$. Since $q$ is not a root of unity, this equality implies $d=1$. The same result states that all derivations in this family are inner whenever $d\neq0$. Therefore, in the ordinary case, this family consists entirely of inner derivations.

Since these two families, together with the inner derivations, exhaust all ordinary derivations of $R$, it follows that $\Der(R)=\operatorname{Inn}(R)+\Bbbk E$. The derivation $E$ is not inner, since it corresponds to $f(h)=1$, whose constant coefficient is nonzero. Then, the sum is direct:
\[
\Der(R)
=
\operatorname{Inn}(R)\oplus\Bbbk E.
\]

Transporting this decomposition through $\Phi$ and using
$\Phi E\Phi^{-1}=-H$, we obtain
\[
\Der(A_q^1(\Bbbk))
=
\operatorname{Inn}(A_q^1(\Bbbk))
\oplus
\Bbbk H.
\]
\end{proof}

The following result classifies the locally finite derivations of the first quantum Weyl algebra $A_q^1(\Bbbk)$ for every $q\in\Bbbk^\ast\setminus\{1\}$,
including the case where $q$ is a root of unity.

\begin{proposition}\label{lfquantum}
Let $q\in\Bbbk^\ast\setminus\{1\}$. A derivation $D\in\Der(A_q^1(\Bbbk))$ is locally finite if and only if
$D=cH$ for some $c\in\Bbbk$. Then,
\[
\LFD(A_q^1(\Bbbk))=\Bbbk H.
\]
Moreover,
\[
\LND(A_q^1(\Bbbk))=\{0\}.
\]
\end{proposition}

\begin{proof}
As shown above, $A_q^1(\Bbbk)\simeq A(\Bbbk[h],q,h-1)$. The defining polynomial $a(h)=h-1$ is not a monomial. Therefore, \cite[Proposition~2.3]{SuarezVivas} implies that every locally finite derivation of $A_q^1(\Bbbk)$ is a scalar multiple of its Eulerian derivation. Under the preceding identification, this Eulerian
derivation is precisely
\[
H(X)=X,
\quad
H(Y)=-Y.
\]

On the other hand, notice that every scalar multiple of $H$ is locally finite, since $H(X^iY^j)=(i-j)X^iY^j$ for every $i,j\in\mathbb N_0$. Consequently, 
\[
\LFD(A_q^1(\Bbbk))=\Bbbk H.
\]

Finally, every locally nilpotent derivation is locally finite. Thus, a locally nilpotent derivation must have the form $cH$. If $c\neq0$, then $(cH)^n(X)=c^nX\neq0$
for every $n\geq0$. Then, $cH$ is not locally nilpotent. Therefore, 
\[
\LND(A_q^1(\Bbbk))=\{0\}.
\]
\end{proof}

Propositions \ref{derquantum} and \ref{lfquantum} yield the following characterization of locally finite derivations in terms of the decomposition of all  derivations.

\begin{corollary}
Assume that $q$ is not a root of unity and write
$D=\ad_w+cH$ according to Proposition \ref{derquantum}. Then $D$ is locally finite if and only if $\ad_w=0$. Equivalently, $D$ is locally finite if and only if $w\in Z(A_q^1(\Bbbk))$.
\end{corollary}

\begin{proof}
By Proposition \ref{lfquantum}, the derivation $D$ is locally finite if and only if $D=c'H$ for some $c'\in\Bbbk$. Since the decomposition
\[
\Der(A_q^1(\Bbbk))=\operatorname{Inn}(A_q^1(\Bbbk))
\oplus
\Bbbk H
\]
is direct, this occurs if and only if the inner component of $D$ vanishes, that is, $\ad_w=0$, which is equivalent to
$w\in Z(A_q^1(\Bbbk))$.
\end{proof}

\section{The differential algebra $A_h$}\label{Thedifferential algebraA_h}

Let $\delta=h(x)\partial_x$ be a derivation of $\Bbbk[x]$, where $h(x)\in\Bbbk[x]$. We denote by
\[
A_h=\Bbbk[x][t;\delta]
\]
the corresponding differential Ore extension. Thus, $A_h$ is the $\Bbbk$-algebra generated by $x$ and $t$ subject to the relation $tx=xt+h(x)$.

We first consider the cases where $h$ is constant. If $h=0$, then $A_h=\Bbbk[x,t]$ is the (commutative) polynomial algebra in two variables and its locally finite derivations are classified, up to conjugation by an automorphism, in
Theorem \ref{v92}. If $h\in\Bbbk^*$, then, the algebra $A_h$ is isomorphic to the first Weyl algebra 
\[
A_1=\Bbbk\langle p,q\mid [p,q]=1\rangle,
\]
whose locally nilpotent and locally finite derivations were described as consequences of Dixmier's classification (\cite[Corollary 6.6 and Theorems 9.1 - 9.2]{Dixmier} and the proof of \cite[Corollary 9.3]{Dixmier}). More precisely, assuming that $\Bbbk$ is algebraically closed of characteristic zero, every locally nilpotent derivation of $A_1$ is conjugate under an automorphism of $A_1$ to a derivation of the form  $\ad_{f(p)}$, where $f(p)\in\Bbbk[p]$, and every locally finite derivation of $A_1$, with is not locally nilpotent, is conjugate to $c\ad_{pq}$, with $c\in\Bbbk^*$. 

Therefore, considering that locally nilpotent and locally finite derivations of $A_h$ are already completely known when $h$ is constant, we assume that $h\in\Bbbk[x]\setminus\Bbbk$ in the remainder of this section.

For every $g(x)\in\Bbbk[x]$, consider the derivation $D_{g(x)}$ of
$A_h$ defined by
\[
D_{g(x)}(x)=0
\quad\text{and}\quad
D_{g(x)}(t)=g(x).
\]
These derivations are analogous to the locally nilpotent derivations
$v(x)\partial_y$ of $\Bbbk[x,y]$ and to the inner derivations
$\ad_{v(x)}$ of the first Weyl algebra $A_1$. It is immediate that
\[
D_{g(x)}\in\LND(A_h).
\]

Locally nilpotent derivations also play an important role in the study of automorphisms. Indeed, for every $\partial\in\LND(A_h)$, the map
\[
\exp(\partial)\colon A_h\longrightarrow A_h,
\quad
\exp(\partial)(a)
=
\sum_{k\geq0}\frac{\partial^k(a)}{k!},
\]
is a well-defined automorphism of $A_h$. Moreover, the automorphisms obtained in this way generate a normal subgroup of $\Aut(A_h)$. 

The locally nilpotent derivations of $A_h$ and its Makar-Limanov invariant $\ml(A_h)$ were determined by Kaygorodov, Lopes, and Mashurov in  \cite{ISF2021}, as follows: 

\begin{proposition}[{\cite[Proposition. 2]{ISF2021}}] \label{prop_localnilp}
Let $h \in \Bbbk[x] \setminus \Bbbk$.
Then  
\[
\LND(A_h)=\{D_{g(x)} \mid g(x) \in \Bbbk[x]\}
\]
and $\ml(A_h)=\Bbbk[x]$.
\end{proposition}

We also recall two results due to Nowicki \cite{Nowicki}, which characterize the derivations of $A_h$ both in the square-free case, that is, when $\gcd(h,h')=1$, and in the non-square-free case.

\begin{theorem}[{\cite[Theorem 10.1]{Nowicki}}] \label{squarefree}
Consider the algebra $A_h$ with $\deg(h)\geq 1$. Assume that $\gcd(h,h')=1$. Then every derivation $D$ of $A_h$ has a unique decomposition
\begin{equation}\label{eq_Dmdc1}
    D=ad_w+\Delta_{s(x)},
\end{equation}
where $ad_w$ is an inner derivation of $A_h$, and  $\Delta_{s(x)}$ is the derivation of $A_h$ given by $\Delta_{s(x)} (x)=0$ and $\Delta_{s(x)} (t)=s(x)$, for $s(x)\in \Bbbk[x]$ and $\deg(s)<\deg(h)$.
Moreover, $D$ is inner if and only if $s(x)=0$.
\end{theorem}

\begin{theorem}[{\cite[Theorem 11.1]{Nowicki}}]\label{Nowicki-nonsquarefree}
Let $h\in \Bbbk[x]$, $\deg h\ge 2$, and $\psi=\gcd(h,h')$.
Then every derivation $D\in\Der(A_h)$ can be uniquely expressed as
\[
D=ad_w+E_H+\Delta_{s(x)},
\]
where $w\in A_h$, $s(x)\in \Bbbk[x]$ with $\deg s<\deg h$,
and $H$ is a special polynomial. Moreover, $D$ is inner if and only if $H=0$ and $s(x)=0$.
\end{theorem}

\subsection{Properties of $A_h$}

In this section we state some relations in the algebra $A_h$ which will be relevant for our purposes:

\begin{lemma}\label{lematx^n} For any $n \geq 0$,
\begin{equation}\label{lematx^n_eq}
tx^n=x^nt+(x^n)'h(x)=x^nt+nx^{n-1} h(x).
\end{equation}

\end{lemma}
\begin{proof}
    The proof is done by induction on $n$. The base cases $n=0$ and $n=1$ are clear.
    Assume the result is valid for $n\geq 0$, then
    \begin{eqnarray*}
    tx^{n+1}&=& tx^nx\\
    &=& (x^nt+nx^{n-1} h(x))x\\
    &=& x^ntx+nx^{n-1} h(x)x\\
    &=& x^n(xt+h(x)) + nx^nh(x)\\
    &=& x^{n+1}t + x^nh(x) + nx^nh(x)\\
    &=& x^{n+1}t + (n+1)x^{n}h(x),
    \end{eqnarray*}
which concludes the proof.
\end{proof}

As a consequence of this Lemma, it is possible to compute the product of $t$ with any polynomial of $\Bbbk[x]$:
\begin{equation}\label{comm_t_fx_eq}
t f(x) = f(x) t + f'(x)h(x).
\end{equation}

In general, the product of $t$ with any element of $A_h$ is more complicated.
However, we can handle the degree in $t$ as follows: recall that for $g \in A_h$, $g \neq 0$, there is a unique decomposition $g = \sum_{i=0}^n g_i(x)t^i$, where $g_i \in \Bbbk[x]$ and $g_n\neq 0$, allowing to set $\deg_t(g):=n$.
Note that $\deg_t(g)=0$ if, and only if, $g \in \Bbbk[x]$. It is also convenient to define $\deg_t(0)=-\infty$.

\begin{corollary}\label{comm_t_g}
    Let $g \in A_h$. Then $t g = g t + h(x)\tilde{g}$, where $\deg_t(\tilde{g}) \leq \deg_t(g)$.
    In particular, $\deg_t( t g ) = \deg_t (g) + 1$, and $\deg_t(\tilde{g})= \deg_t(g)$ if and only if $g = \sum_{i=0}^n g_i(x) t^i$ and $g_n(x) \in \Bbbk[x] \backslash \Bbbk$.

\end{corollary}

\begin{proof}
    Write $g = \sum_{i=0}^n g_i(x)t^i$.
    Then,
    \begin{eqnarray*}
        t g & = & t \left( \sum_{i=0}^n g_i(x)t^i \right) \\
        & = & \sum_{i=0}^n t g_i(x) t^i \\
        & \stackrel{\eqref{comm_t_fx_eq}}{=} & \sum_{i=0}^n (g_i(x) t + g_i'(x) h(x)) t^i \\
        & = & \sum_{i=0}^n g_i(x)t^{i+1} + \sum_{i=0}^n h(x)g_i'(x)t^i \\
        & = & \left( \sum_{i=0}^n g_i(x)t^i \right) t + h(x) \left( \sum_{i=0}^n g_i'(x)t^i \right) \\
        & = &  gt + h(x) \tilde{g},
    \end{eqnarray*}
    where $\tilde{g}= \sum_{i=0}^n g_i'(x)t^i$. Since $\Bbbk$ has characteristic zero, $g_n\neq 0$ and $g_n\notin\Bbbk$ together are equivalent to $g_n'\neq0$, which gives the stated equivalence.
\end{proof}

\medbreak

Next, we also use the standard commutator notation for an algebra $R$:  for $a,b \in R$, we denote $[a,b]:=ab-ba$. Further, given $a\in R$, we denote the inner derivation associated to $a$ by $ad_a : R \to R$, which is defined by $ad_a(b)=ab-ba = [a,b]$.

\begin{lemma}\label{lematix}
For every $i\ge 1$, there exists $u_i\in A_h$ such that
\[
[t^i,x]=i\,h(x)\,t^{i-1}+h(x)u_i
\quad\text{and}\quad
\deg_t(u_i)<i-1.
\]
\end{lemma}

\begin{proof}
For $i=1$ is immediate, since $[t,x]=h(x)$.
Assume that $[t^i,x]=i\,h(x)\,t^{i-1}+h(x)u_i$ for some $u_i\in A_h$ with $\deg_t(u_i)<i-1$. Using
\[
[t^{i+1},x]=t[t^i,x]+[t,x]t^i
\]
and, $th(x)=h(x)(t+h'(x))$, by \ref{comm_t_fx_eq}, we obtain
\[
[t^{i+1},x]
=t\bigl(i\,h(x)\,t^{i-1}+h(x)u_i\bigr)+h(x)t^i
=(i+1)h(x)t^i+h(x)u_{i+1},
\]
for some $u_{i+1}\in A_h$ with $\deg_t(u_{i+1})<i$. This concludes the induction.
\end{proof}

\medbreak

\subsection{The square-free case}

\begin{lemma}\label{onevariable}
Let $p(x)\in \Bbbk[x]$. Then, the derivation $p(x)\partial_x\in \Der(\Bbbk[x])$
is locally finite if and only if $\deg p\le 1$.
\end{lemma}

\begin{proof} The following elementary fact is also recorded in \cite[Example 1.2]{V92}.
\end{proof}

The following result gives a complete description of the locally finite derivations of $A_h$ in the square-free case. It exhibits two distinct situations. If $\deg h\geq2$, every locally finite derivation is locally nilpotent. If $\deg h=1$, an additional term of the form
$\lambda\ad_t$ may occur.

\begin{theorem}\label{clas-square}
Assume that $h\in \Bbbk[x]\setminus \Bbbk$ is square-free.

\begin{enumerate}
\item If $\deg h\ge 2$, then a derivation $D\in \Der(A_h)$ is locally finite if and only if
\[
D=D_{g(x)}
\]
for some $g(x)\in \Bbbk[x]$. In particular, a derivation of $A_h$ is locally finite if and only if it is locally nilpotent.

\item If $\deg h=1$, then a derivation $D\in \Der(A_h)$ is locally finite if and only if
\[
D=\lambda \ad_t+D_{g(x)}
\]
for some $\lambda\in \Bbbk$ and $g(x)\in \Bbbk[x]$.
\end{enumerate}
\end{theorem}

\begin{proof}
By Nowicki's theorem \ref{squarefree}, in the square-free case, every derivation $D$ of $A_h$ can be written as $D=\ad_w+\Delta_{s(x)}$, where $w\in A_h$, $s(x)\in\Bbbk[x]$, $\deg s<\deg h$ and
\[
\Delta_{s(x)}(x)=0,
\quad
\Delta_{s(x)}(t)=s(x).
\]

We prove that, if $D$ is locally finite, then $\deg_t(w)\le 1$. If $w=0$, there is nothing to prove. Thus, assume $w\neq 0$ and write $w=\sum_{i=0}^n f_i(x)t^i$, with $f_n(x)\neq 0$. Assume, by contradiction, that $n\ge 2$. Since $\Delta_{s(x)}(x)=0$, we have
\[
D(x)=[w,x]=\sum_{i=1}^n f_i(x)[t^i,x].
\]

By Lemma \ref{lematix}, $[t^i,x]=i h(x)t^{i-1}+h(x)u_i$, with $\deg_t(u_i)<i-1$. Therefore, $D(x)=n f_n(x)h(x)t^{n-1}+v_1$, where $\deg_t(v_1)<n-1$. We denote $p_1(x):=n f_n(x)h(x)$. Since $h\notin\Bbbk$,  $p_1(x)$ is nonconstant.

We claim that, for every $m\ge 1$, there exist $p_m(x)\in\Bbbk[x]\setminus\Bbbk$ and $v_m\in A_h$ such that
\[
D^m(x)=p_m(x)t^{m(n-1)}+v_m
\]
with $\deg_t(v_m)<m(n-1)$. Moreover, if $d:=\deg h$ and $e:=\deg f_n$, then $\deg p_m=m(d+e-1)+1$. The case $m=1$ follows from the computation above, since $\deg p_1=\deg(f_nh)=e+d=1(d+e-1)+1$.

Assume that the claim holds for some $m\ge 1$. Thus,
\[
D^m(x)=p_m(x)t^{m(n-1)}+v_m,
\quad
\deg_t(v_m)<m(n-1).
\]

Observe that $\deg_t\bigl(\Delta_{s(x)}(a)\bigr)\le \deg_t(a)-1$, for every $a\in A_h$. As a consequence of Lemma \ref{lematix}, for
$f(x),q(x)\in\Bbbk[x]$ and $i,j\ge 0$, we have
\[
[f(x)t^i,q(x)t^j]
=
h(x)(if(x)q'(x)-jq(x)f'(x))t^{i+j-1}
+
\text{terms of lower $t$-degree}.
\]

In particular, commutators lower the expected $t$-degree by at least one. Therefore, the terms $f_i(x)t^i$, with $i<n$, contribute only terms of $t$-degree at most $m(n-1)+i-1\le m(n-1)+n-2$. Also, since $\deg_t(v_m)<m(n-1)$, the term $[w,v_m]$ contributes only terms of $t$-degree at most
$m(n-1)+n-2$. Finally, $\Delta_{s(x)}$ also contributes only terms of $t$-degree strictly smaller than $m(n-1)+n-1=(m+1)(n-1)$.

Thus, the term of highest possible $t$-degree in $D^{m+1}(x)$ comes only from $[f_n(x)t^n,p_m(x)t^{m(n-1)}]$. Using the equation above, we obtain
\[
[f_n(x)t^n,p_m(x)t^{m(n-1)}]
=
\]
\[
h(x)\bigl(n f_n(x)p_m'(x)-m(n-1)p_m(x)f_n'(x)\bigr)
t^{m(n-1)+n-1}
+
\text{terms of lower $t$-degree}.
\]

Since $m(n-1)+n-1=(m+1)(n-1)$, we obtain
\[
D^{m+1}(x)
=
p_{m+1}(x)t^{(m+1)(n-1)}+v_{m+1},
\]
where $p_{m+1}(x)=h(x)(n f_n(x)p_m'(x)-m(n-1)p_m(x)f_n'(x))$
and, additionally, $\deg_t(v_{m+1})<(m+1)(n-1)$. It remains to show that $p_{m+1}(x)$ is nonzero and has the desired degree. Let $M:=\deg p_m$. By the induction hypothesis, $M=m(d+e-1)+1$. The highest-degree term of
\[
n f_n(x)p_m'(x)-m(n-1)p_m(x)f_n'(x)
\]
has coefficient proportional to $nM-m(n-1)e$. Since $M=m(d+e-1)+1$, $nM-m(n-1)e=m(n(d-1)+e)+n$, and this is nonzero because $m\ge 1$, $n\ge 2$, $d\ge 1$, and $e\ge 0$. Therefore, $\deg(n f_n p_m'-m(n-1)p_m f_n')=M+e-1$.

Multiplying by $h(x)$, we obtain
\[
\deg p_{m+1}=d+M+e-1=(m+1)(d+e-1)+1,
\]
and then $p_{m+1}$ is nonzero and nonconstant, this proves the induction.

Consequently, $\deg_t(D^m(x))=m(n-1)$ for every $m\ge 1$. Since $n\ge 2$, these degrees are unbounded. Therefore, $\{D^m(x)\}_{m\ge 0}$ span an infinite-dimensional vector space. This contradicts the local finiteness of $D$.

Therefore, we write $w=u(x)t+v(x)$, with $u(x),v(x)\in\Bbbk[x]$. Let $b(x)\in\Bbbk[x]$. Since $\Delta_{s(x)}(b(x))=0$, we have
\[
D(b(x))
=
[u(x)t+v(x),b(x)]
=
u(x)[t,b(x)].
\]

By Equation \eqref{comm_t_fx_eq}, $[t,b(x)]=b'(x)h(x)$. Thus,
$D(b(x))=u(x)h(x)b'(x)$. Then, the restriction of $D$ to $\Bbbk[x]$ is
\[
D|_{\Bbbk[x]}=u(x)h(x)\partial_x.
\]

Since $D$ is locally finite, its restriction to the $D$-stable subalgebra $\Bbbk[x]$ is locally finite. By Lemma \ref{onevariable}, we obtain $\deg(uh)\le 1$.

Assume first that $\deg h\ge 2$. If $u\neq 0$, then $\deg(uh)=\deg u+\deg h\ge 2$, contradicting $\deg(uh)\le 1$. Thus, $u=0$. Consequently, $D=\ad_{v(x)}+\Delta_{s(x)}$.

Notice that
\[
\ad_{v(x)}(t)=v(x)t-tv(x)=-h(x)v'(x),
\]
and then $\ad_{v(x)}=\Delta_{-h(x)v'(x)}$. Thus,
\[
D=\Delta_{s(x)-h(x)v'(x)}=D_{s(x)-h(x)v'(x)}.
\]

Therefore, every locally finite derivation is of the form $D_{g(x)}$. Conversely, every derivation of the form $D_{g(x)}$ is locally nilpotent by Proposition \ref{prop_localnilp}, and then locally finite. So, if
$\deg h\ge 2$, the locally finite derivations are precisely the locally nilpotent derivations: this proves ${\rm(1)}$.

Assume that $\deg h=1$. Since $\deg(uh)\le 1$, we obtain $u\in\Bbbk$. Write $u=\lambda$ and then
\[
D=\lambda\ad_t+\ad_{v(x)}+\Delta_{s(x)}.
\]

As above, $\ad_{v(x)}=\Delta_{-h(x)v'(x)}$. Thus,
\[
D=\lambda\ad_t+\Delta_{s(x)-h(x)v'(x)}
=\lambda\ad_t+D_{g(x)},
\]
for some $g(x)\in\Bbbk[x]$. It remains to prove the converse. Let $D=\lambda\ad_t+D_{g(x)}$, with $\lambda\in\Bbbk$ and $g(x)\in\Bbbk[x]$. We prove that $D$ is locally finite. Write $h(x)=\alpha x+\beta$, with $\alpha\in\Bbbk^*$, $\beta \in \Bbbk$, and denote $z:=h(x)$. Then, $\Bbbk[x]=\Bbbk[z]$ and, by Equation \eqref{comm_t_fx_eq}, $tz=zt+\alpha z$. Equivalently, for every $F(z)\in\Bbbk[z]$,
\[
tF(z)=F(z)t+\alpha zF'(z).
\]

Write $g(x)=G(z)$, for some $G(z)\in\Bbbk[z]$. If $G=0$, denote $\ell=0$; otherwise, $\ell:=\deg G$. Notice that
\[
\ad_t(z)=\alpha z,
\quad
\ad_t(t)=0,
\]
and then $\ad_t(z^it^j)=\alpha iz^it^j$. We claim that, for every $j\ge 1$,
\[
D_{g(x)}(t^j)
\in
\operatorname{Span}_\Bbbk
\{z^r t^s\mid 0\le s\le j-1,\ 0\le r\le (j-s)\ell\}.
\]

For $j=1$, this follows from $D_{g(x)}(t)=g(x)=G(z)$, because $\deg G\le \ell$. Assume that the claim holds for some $j\ge 1$. Then,
\[
D_{g(x)}(t^{j+1})=D_{g(x)}(t^j)t+t^jG(z).
\]

By the induction hypothesis, every term appearing in $D_{g(x)}(t^j)$ is a linear combination of monomials $z^r t^s$ with
$0\le s\le j-1$ and $0\le r\le (j-s)\ell$. Multiplying such a monomial by $t$, we obtain $z^r t^{s+1}$. Since
$(j-s)\ell=((j+1)-(s+1))\ell$, the term $D_{g(x)}(t^j)t$ satisfies the desired bound for $j+1$. For the second term, we show by induction on $j$ that $t^jG(z)=\sum_{s=0}^j H_s(z)t^s$ with $\deg H_s\le\ell$ for every $s$: for $j=0$ this is $G(z)$ itself, of degree $\ell$; assuming it for $j$, we get
\[
t^{j+1}G(z)= 
\]
\[
t\Bigl(\sum_{s=0}^jH_s(z)t^s\Bigr)=\sum_{s=0}^j\bigl(H_s(z)t+\alpha zH_s'(z)\bigr)t^s=\sum_{s=0}^jH_s(z)t^{s+1}+\sum_{s=0}^j\alpha zH_s'(z)t^s,
\]
and since $\deg H_s\le\ell$ implies $\deg(zH_s')\le\ell$ as well, both families of coefficients appearing on the right have degree at most $\ell$, proving the claim for $j+1$.
Since $0\le s\le j$, we have $\ell\le ((j+1)-s)\ell$. Therefore, the term $t^jG(z)$ also satisfies the desired bound for $j+1$, this proves the claim.

Let $c=\sum_{i,j} c_{ij}z^it^j\in A_h$. Choose $N\ge 0$ such that $j\le N$ for every monomial $z^it^j$ appearing in $c$.
Also, choose $M\ge 0$ such that every monomial appearing in $c$ satisfies $i\le M+(N-j)\ell$. Denote
\[
\mathcal V_{M,N}:=
\operatorname{Span}_\Bbbk
\{z^it^j\mid 0\le j\le N,\ 0\le i\le M+(N-j)\ell\}.
\]

This is a finite-dimensional vector space containing $c$ such that $\ad_t(\mathcal V_{M,N})\subseteq \mathcal V_{M,N}$. Moreover, by the claim above, if $z^it^j\in \mathcal V_{M,N}$, then
\[
D_{g(x)}(z^it^j)=z^iD_{g(x)}(t^j)
\]
belongs to $\mathcal V_{M,N}$. Indeed, every term arising from $D_{g(x)}(t^j)$ has the form $z^r t^s$ with $0\le s\le j-1$, and $0\le r\le (j-s)\ell$. Then,
\[
i+r
\le
M+(N-j)\ell+(j-s)\ell
=
M+(N-s)\ell.
\]

Therefore, $D_{g(x)}(\mathcal V_{M,N})\subseteq \mathcal V_{M,N}$ and then $D(\mathcal V_{M,N})\subseteq \mathcal V_{M,N}$. Since $c\in \mathcal V_{M,N}$, the iterates $D^m(c)$ all belong to the finite-dimensional
space $\mathcal V_{M,N}$. So, $D$ is locally finite, this proves ${\rm(2)}$.
\end{proof}

\subsection{Locally finite sets of derivations}

We recall a result of Chitayat, Daigle, and Regeta
\cite{CDR} on locally finite sets of derivations, which will be used to establish a consequence of the preceding classification.

Let $B$ be a $\Bbbk$-algebra. A subset $\Delta\subseteq\Der(B)$ is said to be locally finite if, for every $b\in B$, there exists a finite-dimensional $\Bbbk$-subspace $U_b\subseteq B$ such that $b\in U_b$
and $D(U_b)\subseteq U_b$ for every $D\in\Delta$. The set $\Delta$ is said to be weakly locally finite if every finite subset of $\Delta$ is locally finite. Notice that weak local finiteness is stronger than requiring each element of $\Delta$ to be a locally finite derivation. Indeed, for a finite family of derivations, one requires the existence of a common finite-dimensional invariant subspace.

Following \cite{CDR}, a $\Bbbk$-algebra $B$ is called
derivation-finite if there exists a finite subset $S\subseteq B$ such that the only derivation vanishing on every element of $S$ is the zero derivation. Every finitely generated $\Bbbk$-algebra is derivation-finite. In particular, since $A_h$ is generated by $x$ and $t$, the algebra $A_h$ is derivation-finite.

Assume that $B$ is derivation-finite. By \cite[Corollary 1(b)]{CDR}, if $L$ is a solvable Lie subalgebra of $\Der(B)$ satisfying
\[
L\subseteq\LFD(B),
\]
then $L$ is weakly locally finite. Moreover, every finitely generated Lie subalgebra of a weakly locally finite Lie algebra of derivations is a locally finite set of derivations; see \cite[Remark 3.3(b)]{CDR}.

Variants of this problem were studied in \cite{Zaidenberg} in the case where $B$ is the algebra of regular functions on an affine variety. In particular, \cite{Zaidenberg} proves that, if $\Bbbk$ is an algebraically closed field of characteristic zero and $L$ is a solvable Lie subalgebra
of ${\rm Lie}(\Aut(\mathbb A^2))$ generated by locally finite derivations, then $L$ is weakly locally finite.

In general, the set $\LFD(B)$ need not be a Lie subalgebra of $\Der(B)$, since it may fail to be closed under addition or under the Lie bracket. Therefore, in order to apply
\cite[Corollary 1(b)]{CDR} with $L=\LFD(A_h)$, we must first verify that $\LFD(A_h)$ is a solvable Lie subalgebra of $\Der(A_h)$. Consequently, every finitely generated Lie subalgebra of $\LFD(A_h)$ is locally finite as a set of derivations; see \cite[Remark 3.3(b)]{CDR}.

\begin{corollary}\label{corAh}
Let $A_h$, where $h\in\Bbbk[x]\setminus\Bbbk$ is square-free. Then, $\LFD(A_h)$ is a solvable Lie subalgebra of $\Der(A_h)$. Moreover, $\LFD(A_h)$ is weakly locally finite, but it is not locally finite as a set of derivations.
\end{corollary}

\begin{proof}
Since $A_h$ is generated as a $\Bbbk$-algebra by $x$ and $t$, it is derivation-finite. We first show that $\LFD(A_h)$ is a solvable Lie subalgebra of
$\Der(A_h)$.

Suppose that $\deg h\geq2$. By the classification of locally finite derivations of $A_h$, Theorem \ref{clas-square}, we have
\[
\LFD(A_h)
=
\{D_g\mid g\in\Bbbk[x]\}.
\]

Furthermore, $[D_f,D_g]=0$, for all $f,g\in\Bbbk[x]$. Therefore, $\LFD(A_h)$ is abelian and, in particular, solvable.

Suppose that $\deg h=1$. In this case,
\[
\LFD(A_h)
=
\{\lambda\ad_t+D_g
\mid
\lambda\in\Bbbk,\ g\in\Bbbk[x]\}.
\]

Since $[\ad_t,D_f]=D_{hf'}$, we obtain $[\lambda\ad_t+D_g,\mu\ad_t+D_f]=D_{h(\lambda f'-\mu g')}$, for all $\lambda,\mu\in\Bbbk$ and $f,g\in\Bbbk[x]$. Thus,
\[
[\LFD(A_h),\LFD(A_h)]
\subseteq
\{D_g\mid g\in\Bbbk[x]\}.
\]

Since $[D_f,D_g]=0$ for all $f,g\in\Bbbk[x]$, it follows that $\LFD(A_h)^{(2)}=0$. Thus, $\LFD(A_h)$ is solvable.

In both cases, $\LFD(A_h)$ is a solvable Lie subalgebra of
$\Der(A_h)$. Therefore, \cite[Corollary 1(b)]{CDR} implies that $\LFD(A_h)$ is weakly locally finite.

It remains to show that $\LFD(A_h)$ is not locally finite as a set. Suppose, by contradiction, that it is locally finite. Then, there exists a finite-dimensional subspace $U_t\subseteq A_h$ such that $t\in U_t$ and $D(U_t)\subseteq U_t$, for every $D\in\LFD(A_h)$. For every $g(x)\in\Bbbk[x]$, the derivation $D_g$ belongs to $\LFD(A_h)$ and satisfies $D_g(t)=g(x)$. Then, $g(x)\in U_t$ for every $g(x)\in\Bbbk[x]$. This implies $\Bbbk[x]\subseteq U_t$, contradicting the finite-dimensionality of $U_t$. Therefore, $\LFD(A_h)$ is not locally finite as a set of derivations.

\end{proof}

\subsection{The singular case}

In this subsection we assume that $h\in\Bbbk[x]\setminus\Bbbk$ is not square-free. We also denote,
\[
\psi:=\gcd(h,h')
\quad \text{and} \quad
g:=\frac{h}{\psi}.
\]

Let $N=\deg h$, and let $r$ be the number of pairwise distinct roots of $h$. Then, $\deg\psi=N-r$. We say that a polynomial $H \in A_h$ is \emph{special} if it is of the form
$$ H=h_nt^n+\ldots+h_1t^1,$$
where $n \ge 1$, $h_i \in \Bbbk[x]$ and $\deg (h_i) < N-r$, for all $i= 1,\ldots,n$. Given a special polynomial $H$, the linear map $E_H: A_h \to A_h$ defined as
$$ E_H(f)=\left[ \tfrac{1}{\psi}H,f \right], \quad f\in A_h, $$
is a well-defined derivation of $A_h$ (see \cite[Prop. 2.1 and 8.2]{Nowicki}). Moreover, $E_H$ is an inner derivation of $A_h$ if and only if $H=0$.

The bracket in $E_H$ is computed in the localization
\[
\mathcal B:=A_h[S^{-1}],
\quad
S=\{\psi^n\mid n\ge 0\}.
\]

By Nowicki's theorem \ref{Nowicki-nonsquarefree}, every derivation $D\in\Der(A_h)$ can be uniquely written as
\[
D=\ad_w+E_H+\Delta_{s(x)},
\]
where $w\in A_h$, $H$ is special, and $s(x)\in\Bbbk[x]$ satisfies $\deg s<\deg h$.

\begin{theorem}\label{clas-nonsquare} Let $h\in\Bbbk[x]$, $\deg (h)\ge 1$, be a non-square-free polynomial. Then,
\begin{enumerate}
\item If $\deg g\ge 2$, then a derivation $D\in\Der(A_h)$ is locally finite if
and only if
\[
D=D_{p(x)}
\]
for some $p(x)\in\Bbbk[x]$. In particular, a derivation of $A_h$ is locally
finite if and only if it is locally nilpotent.

\item If $\deg g=1$, then a derivation $D\in\Der(A_h)$ is locally finite if and
only if
\[
D=\lambda E_t+D_{p(x)}
\]
for some $\lambda\in\Bbbk$ and $p(x)\in\Bbbk[x]$.
\end{enumerate}
\end{theorem}

\begin{proof}
By Theorem \ref{Nowicki-nonsquarefree}, write $D=\ad_w+E_H+\Delta_{s(x)}$, where $w\in A_h$, $H$ is special, and $\deg s<\deg h$. Recall that
\[
E_H(f)=\left[\frac{1}{\psi}H,f\right],
\quad f\in A_h,
\]
where the bracket is taken in the localization
$\mathcal B=A_h[S^{-1}]$, where $S=\{\psi^m\mid m\geq0\}$.
Since $h\partial_x$ extends to $\Bbbk[x,\psi^{-1}]$, notice that we also may consider $\mathcal B=\Bbbk[x,\psi^{-1}][t;h(x)\partial_x]$, and the $t$-degree extends naturally to $\mathcal B$. We also denote $W:=w+\frac{1}{\psi}H$.
Then, for every $f\in A_h$,
\[
D(f)=[W,f]+\Delta_{s(x)}(f).
\]

We claim that local finiteness of $D$ implies
$\deg_t(W)\leq1$. Write
\[
w=\sum_i u_i(x)t^i,
\quad
H=\sum_i h_i(x)t^i.
\]

Thus, the coefficient of $t^i$ in $W$ is
\[
\frac{\psi(x)u_i(x)+h_i(x)}{\psi(x)}.
\]
Since $H$ is special, every nonzero $h_i$ satisfies
$\deg h_i<\deg\psi$. Then, $\psi u_i+h_i=0$ if and only if $u_i=h_i=0$. Therefore,
\[
\deg_t(W)=\max\{\deg_t(w),\deg_t(H)\}.
\]

Suppose that $n:=\deg_t(W)\geq2$, and write the leading coefficient of $W$ as $c(x)/\psi(x)$, with $c(x)\neq0$. The leading $t$-degree argument used in the proof of Theorem \ref{clas-square} extends to
the present localization. Indeed, for $a(x),b(x)\in\Bbbk(x)$,
\[
[a(x)t^i,b(x)t^j]
=
h(x)\bigl(i\,a(x)b'(x)-j\,b(x)a'(x)\bigr)t^{i+j-1}
+
\text{terms of lower $t$-degree}.
\]
Consequently, an induction gives
\[
D^m(x)
=
p_m(x)t^{m(n-1)}
+
\text{terms of lower $t$-degree},
\]
where $p_m(x)\neq0$ for every $m\geq1$. More precisely, if
\[
d:=\deg g,\quad e:=\deg c,\quad \rho:=\deg\psi,
\]
then $\deg p_m=m(d+e-1)+1$. At the induction step, the coefficient of the highest possible
$x$-degree is proportional to $m\bigl(n(d-1)+e+(n-1)\rho\bigr)+n$. Since $d\geq1$, $\rho\geq1$, $e\geq0$, and $n\geq2$, this is a
positive integer, and hence it is nonzero because $\operatorname{char}(\Bbbk)=0$. Thus, $\deg_t(D^m(x))=m(n-1)$, for every $m\geq1$, contradicting the local finiteness of $D$. Therefore, $\deg_t(W)\leq1$ and, then, $w=u(x)t+v(x)$ and $H=r(x)t$, where $u(x),v(x),r(x)\in\Bbbk[x]$ and $\deg r<\deg\psi$. For every $b(x)\in\Bbbk[x]$, using \eqref{comm_t_fx_eq} and $h=\psi g$, we obtain
\[
\begin{aligned}
D(b(x))
&=
[u(x)t+v(x),b(x)]
+
\left[\frac{r(x)}{\psi(x)}t,b(x)\right]\\
&=
g(x)\bigl(\psi(x)u(x)+r(x)\bigr)b'(x).
\end{aligned}
\]

Then, $\Bbbk[x]$ is $D$-stable and $D|_{\Bbbk[x]}
=g(x)\bigl(\psi(x)u(x)+r(x)\bigr)\partial_x$. Since $D$ is locally finite, its restriction to $\Bbbk[x]$ is
locally finite. By Lemma \ref{onevariable}, the polynomial
$g(\psi u+r)$ is either zero or has degree at most one.

Suppose first that $\deg g\geq2$. Then, necessarily
$\psi u+r=0$. Since $\deg r<\deg\psi$, this implies
$u=0$ and $r=0$. Thus,
\[
D=\ad_{v(x)}+\Delta_{s(x)}.
\]
Since $\ad_{v(x)}(x)=0$ and $\ad_{v(x)}(t)=-h(x)v'(x)$, we obtain $D=D_{s(x)-h(x)v'(x)}$. Conversely, every derivation $D_{p(x)}$ is locally nilpotent by
Proposition \ref{prop_localnilp}, and then locally finite. By the same proposition, these are precisely the locally nilpotent derivations of $A_h$. This proves $\rm(1)$.

Assume now that $\deg g=1$. Then, $\psi u+r=\lambda$,
for some $\lambda\in\Bbbk$. Since $\deg r<\deg\psi$, we again obtain $u=0$ and $r=\lambda$. Therefore,
\[
D=\lambda E_t+\ad_{v(x)}+\Delta_{s(x)}
=\lambda E_t+D_{p(x)},
\]
for some $p(x)\in\Bbbk[x]$.

It remains to prove the converse. Since $\deg g=1$, the polynomial $h$ has a unique distinct root. Thus, putting $z=x-a$ for a suitable $a\in\Bbbk$, we may write
$h=\alpha z^N$, where $N=\deg h\geq2$, for some $\alpha\in\Bbbk^*$. Then, for suitable $\beta,\gamma\in\Bbbk^*$, $\psi=\beta z^{N-1}$ and $g=\gamma z$. In particular, $t$ is special, and
\[
E_t(z)=\gamma z,
\quad
E_t(t)=\gamma(N-1)t.
\]

Let $D=\lambda E_t+D_{p(x)}$ and write $p(x)=P(z)$. Choose $L\geq N-1$ such that, if $P\neq0$, $L\geq\deg P$.
For $M\geq0$, we denote
\[
\mathcal V_M=\operatorname{Span}_{\Bbbk}
\{z^it^j\mid i+Lj\leq M\}.
\]
Each $\mathcal V_M$ is finite-dimensional. Moreover,
$tz=zt+\alpha z^N$, and $N\leq L+1$, so the spaces $\mathcal V_M$ form a multiplicative filtration of $A_h$.

Since $E_t(z)=\gamma z$ and $E_t(t)=\gamma(N-1)t$, we have
$E_t(\mathcal V_M)\subseteq\mathcal V_M$. Also,
\[
D_{p(x)}(z)=0,
\quad
D_{p(x)}(t)=P(z),
\]
and $\deg P\leq L$ whenever $P\neq0$. Then, by the Leibniz rule, $D_{p(x)}(\mathcal V_M)\subseteq\mathcal V_M$. Therefore, $D(\mathcal V_M)\subseteq\mathcal V_M$. Since every element of $A_h$ belongs to some finite-dimensional $\mathcal V_M$, the derivation $D$ is locally finite. This proves $\rm(2)$.
\end{proof}

Thus, as in the square-free case, the classification above leads to the same phenomenon for locally finite sets of derivations.

\begin{corollary}
Let $h\in\Bbbk[x]\setminus\Bbbk$ be non-square-free, and let $\psi=\gcd(h,h')$ and $g=\frac{h}{\psi}$. Then, $\LFD(A_h)$ is a solvable Lie subalgebra of $\Der(A_h)$ and is weakly locally finite. In particular, every finite subset of $\LFD(A_h)$ is locally finite as a set of derivations. However, $\LFD(A_h)$ is not locally finite as a set of derivations.
\end{corollary}

\begin{proof}
The proof is analogous to that of Corollary \ref{corAh}. If $\deg g\geq2$, then Theorem \ref{clas-nonsquare} gives
\[
\LFD(A_h)=\{D_p\mid p\in\Bbbk[x]\},
\]
which is abelian. If $\deg g=1$, then
\[
\LFD(A_h)
=
\{\lambda E_t+D_p
\mid
\lambda\in\Bbbk,\ p\in\Bbbk[x]\}.
\]
Moreover, as in the proof of Theorem \ref{clas-nonsquare},
$[E_t,D_p]=D_q$, for some $q\in\Bbbk[x]$. Since $[D_p,D_q]=0$ for all $p,q\in\Bbbk[x]$, it follows that $\LFD(A_h)$ is solvable.

The weak local finiteness, as well as the fact that $\LFD(A_h)$ is not locally finite as a set of derivations, follows exactly as in Corollary \ref{corAh}.
\end{proof}

\bibliographystyle{abbrv}

\bibliography{referencias}

@article {V92,
    AUTHOR = {van den Essen, Arno},
     TITLE = {Locally finite and locally nilpotent derivations with
              applications to polynomial flows and polynomial morphisms},
   JOURNAL = {Proc. Amer. Math. Soc.},
  FJOURNAL = {Proceedings of the American Mathematical Society},
    VOLUME = {116},
      YEAR = {1992},
    NUMBER = {3},
     PAGES = {861--871},
      ISSN = {0002-9939,1088-6826},
   MRCLASS = {13B10 (14E09 34A99 34C99)},
  MRNUMBER = {1111440},
MRREVIEWER = {Gary\ H.\ Meisters},
       DOI = {10.2307/2159458},
       URL = {https://doi.org/10.2307/2159458},
}

@Article{SuarezVivas,
  author    = {Suárez-Alvarez, Mariano and Vivas, Quimey},
  journal   = {Journal of Algebra},
  title     = {Automorphisms and isomorphisms of quantum generalized Weyl algebras},
  year      = {2015},
  issn      = {0021-8693},
pages     = {540-552},
  volume    = {424},
  doi       = {10.1016/j.jalgebra.2014.08.045},
  publisher = {Elsevier},
}

@Article{bre,
  author    = {Almulhem, Munerah and Brzeziński, Tomasz},
  journal   = {Journal of Algebra},
  title     = {Skew derivations on generalized Weyl algebras},
  year      = {2018},
  volume    = {493},
  pages     = {194--235},
  doi       = {10.1016/j.jalgebra.2017.09.018},
  publisher = {Elsevier},
}

@article{BLO2015,
  title={A parametric family of subalgebras of the {W}eyl algebra {I}. {S}tructure and automorphisms},
  author={Georgia Benkart and Samuel A. Lopes and Matthew Ondrus},
  journal={Transactions of the American Mathematical Society},
  year={2015},
  volume={367},
number={3},
  pages={1993--2021},
  url={https://api.semanticscholar.org/CorpusID:119174319}
}

@article{Re1968,
author={R. Rentschler},
title={Op\'{e}rations du groupe additif sur le plan affine},
journal={C. R. Acad. Sci. Paris Ser. A--B},
year={1968},
volume={267},
pages={384--387},
URL={https://cir.nii.ac.jp/crid/1573105974149033600},
}

@article{ISF2021,
author = {Ivan Kaygorodov and Samuel A. Lopes and Farukh Mashurov},
journal = {Communications in Mathematics},
language = {eng},
number = {2},
pages = {269--279},
publisher = {University of Ostrava},
title = {Actions of the additive group ${G}_a$ on certain noncommutative deformations of the plane},
url = {http://eudml.org/doc/297521},
volume = {29},
year = {2021},
}

@Article{Alev,
  author    = {Alev, J. and Chamarie, M.},
  journal   = {Communications in Algebra},
  title     = {Derivations et automorphismes de quelques algebras quantiques},
  year      = {1992},
  issn      = {1532-4125},
  month     = jan,
  number    = {6},
  pages     = {1787--1802},
  volume    = {20},
  doi       = {10.1080/00927879208824431},
  publisher = {Informa UK Limited},
}

@Article{CDR,
  author  = {Chitayat, Michael and Daigle, Daniel and Regeta, Andriy},
  title   = {Locally finite sets of derivations},
  journal = {Proc. Amer. Math. Soc.},
  year    = {2026},
  note    = {In press},
  doi     = {10.1090/proc/17868},
  eprint  = {2606.08677},
  archiveprefix = {arXiv},
  primaryclass  = {math.AC},
}

@Article{Zaidenberg,
  author        = {Zaidenberg, Mikhail},
  journal       = {arXiv preprint},
  title         = {Locally finite solvable Lie algebras of derivations},
  year          = {2026},
  month         = apr,
  archiveprefix = {arXiv},
  doi           = {10.48550/arXiv.2604.02864},
  eprint        = {2604.02864},
  primaryclass  = {math.AG},
  url           = {https://arxiv.org/abs/2604.02864},
}

@article{Nowicki,
author = {Andrzej Nowicki},
title = {Derivations of {O}re Extensions of the Polynomial Ring in One Variable},
journal = {Communications in Algebra},
volume = {32},
number = {9},
pages = {3651--3672},
year = {2004},
publisher = {Taylor \& Francis},
doi = {10.1081/AGB-120039638},
URL = {https://doi.org/10.1081/AGB-120039638},
eprint = {https://doi.org/10.1081/AGB-120039638}
}

@article{SBVA2025,
author = {A. Santana and R. Baltazar and R. Vinciguerra and W. Araujo},
title = {On isotropy groups of quantum plane},
journal = {Journal of Pure and Applied Algebra},
volume = {229},
number = {11},
pages = {},
year = {2025},
publisher = {},
doi = {},
URL = {https://doi.org/10.1016/j.jpaa.2025.108095},
eprint = {https://doi.org/10.1016/j.jpaa.2025.108095}
}

@article{Dixmier,

author = {J. Dixmier},

title = {Sur les algèbres de Weyl},

journal = {Bull. Soc. Math. France},

volume ={96},

year = {1968}, 

pages = {209–242}
}

\end{document}